\documentclass{amsart}
\long\def\symbolfootnote[#1]#2{\begingroup%
\def\thefootnote{\fnsymbol{footnote}}\footnote[#1]{#2}\endgroup}
\newtheorem{thm}{Theorem}[section]

\theoremstyle{definition}
\newtheorem{defn}[thm]{Definition}
\theoremstyle{remark}
\newtheorem{rem}[thm]{Remark}

\begin{document}

\title[Curvature properties of pseudo-sphere bundles]{Curvature properties of pseudo-sphere bundles over paraquaternionic manifolds}

\author[G.E. V\^{\i}lcu, R.C. Voicu]{Gabriel Eduard V\^{\i}lcu, Rodica Cristina Voicu}

\date{}
\maketitle

\abstract In this paper we obtain several curvature properties of the
twistor and reflector spaces of a paraquaternionic K\"{a}hler
manifold and prove the existence of both positive and negative mixed
3-Sasakian
structures in a principal $SO(2,1)$-bundle over a paraquaternionic K\"{a}hler manifold. \\[1mm]

\endabstract

\section{Introduction}

\symbolfootnote[0]{This work was partially supported by CNCSIS -
UEFISCSU, project PNII - IDEI code 8/2008, contract no. 525/2009.}

An almost paraquaternionic structure on a smooth manifold $M$ is a
rank 3-subbundle of $\text{End}(TM)$ which is locally spanned by two
almost product structures and one almost complex structure which
satisfy relations of anti-commutation. Under certain conditions on
the holonomy group of a metric adapted to such a structure, one
arrives to the concept of paraquaternionic K\"{a}hler manifold. If
$M$ is a manifold endowed with an almost paraquaternionic structure,
then two kinds of unit pseudo-sphere bundle can be considered
\cite{IMZ}: the twistor space $Z^{-}(M)$ of $M$ (the unit
pseudo-sphere bundle with fibre the 2-sheeted hyperboloid
$S^{2}_{1}(-1)$) and the reflector space $Z^{+}(M)$ of $M$ (the unit
pseudo-sphere bundle with fibre the 1-sheeted hyperboloid
$S^{2}_{1}(1)$). These spaces of great importance in theoretical
physics (see e.g. \cite{P,W}) have been studied recently by several
authors \cite{AC,AMT,BDM,DJS,DAV,IZ,SCF}. We note that if $M$ is a
paraquaternionic K\"{a}hler manifold, then the twistor space
$Z^{-}(M)$ admits two canonical almost complex structures $I_1$ and
$I_2$ ($I_1$ integrable in dimension greater than 4, $I_2$ never
integrable) and the reflector space $Z^{+}(M)$ admits two canonical
almost para-complex structures $P_1$ and $P_2$ ($P_1$ integrable in
dimension greater than 4, $P_2$ never integrable). Moreover, both
$Z^{+}(M)$ and $Z^{-}(M)$ carries a natural 1-parameter family of
pseudo-Riemannian metrics $h_{t}$, $t\neq 0$, such that the
canonical projection $ \pi_{\pm}:Z^{\pm}(M)\rightarrow M$ is a
pseudo-Riemannian submersion \cite{AC}. After giving some
introductory results, in Section 2, we study the twistor and
reflector spaces of a paraquaternionic K\"{a}hler manifold $M$ and
the almost pseudo-Hermitian geometry of $(Z^{-}(M), I_{i}, h_{t})$
and $(Z^{+}(M), P_{i}, h_{t})$, $i = 1, 2$. In Section 3 we give
explicit formulas of the pseudo-Riemannian curvature tensors of
$h_{t}$ in terms of the curvature of $M$ and study its properties in
details. In particular, we determine the affiliation to Gray and
para-Gray classes of $(Z^{-}(M), I_{i}, h_{t})$ and $(Z^{+}(M),
P_{i}, h_{t})$, $i = 1, 2$.

On the other hand, the counterpart in odd dimension of a
paraquaternionic structure was introduced in \cite{IMV}. It is
called a mixed 3-structure, and appears in a natural way on
lightlike hypersurfaces in paraquaternionic K\"{a}hler manifolds.
Such hypersurfaces inherit two almost paracontact structures and an
almost contact structure, satisfying analogous conditions to those
satisfied by almost contact 3-structures \cite{KUO}. This concept
has been refined in \cite{CP}, where the authors have introduced
positive and negative  metric mixed 3-structures.

W.M. Boothby and H.C. Wang gave in \cite{BW} a fibering of a compact
manifold $M$ with a regular contact structure as a principal bundle
over a symplectic manifold $M'$ with 1--dimensional toroidal group
$A$. Associating a Riemannian metric $g'$ with the symplectic
structure, $M'$ is an almost K\"{a}hlerian manifold and consequently
$M$ is a contact metric one. Y. Hatakeyama proved in \cite{HAT} that
there exists a circle bundle over an almost complex manifold that
admits a normal contact metric structure. Moreover, K. Sakamoto was
proved in \cite{SAK} that every quaternionic K\"{a}hler manifold
admits a principal bundle $P$ over it, whose structure group is
$\text{SO}(3)$, while Konishi \cite{KO} constructed a Sasakian
3-structure in $P$, which is canonically associated with a given
quaternion K\"{a}hler structure of positive scalar curvature. It can
be easily checked that the Sakamoto's construction remains valid if
we consider a paraquaternionic K\"{a}hler manifold, but now the
structure group is $\text{SO}(2,1)$. The main purpose of the last
Section is to prove that the principal bundle associated with a
paraquaternionic K\"{a}hler manifold $(M,\sigma,g)$ admits both a
positive and a negative mixed 3-Sasakian structure, provided that
$M$ is of non-zero scalar curvature. Moreover, we investigate the
existence of non-trivial Einstein metrics in the canonical variation
of a pseudo-Riemannian submersion which projects a metric mixed
3-contact structure onto a paraquaternionic K\"{a}hler structure.

\section{Preliminaries}

Let $\widetilde{\mathbb{H}}$ be the algebra of paraquaternions and
identify $\widetilde{\mathbb{H}}^{n}=\mathbb{R}^{4n}$. It is well
known that $\widetilde{\mathbb{H}}$ is a real Clifford algebras with
$\widetilde{\mathbb{H}}=C(1,1)\cong C(0,2)$. In fact
$\widetilde{\mathbb{H}}$ is generated by the unity $1$ and
generators $J^{0}_{1},J^{0}_{2},J^{0}_{3}$ satisfying the
paraquaternionic identities
\begin{equation}\label{p1}
      (J^{0}_{1})^{2}=(J^{0}_{2})^{2}=-(J^{0}_{3})^{2}=1,\  J^{0}_{1}J^{0}_{2}=-J^{0}_{2}J^{0}_{1}=J^{0}_{3}.
\end{equation}

We assume that $\widetilde{\mathbb{H}}$ acts on
$\widetilde{\mathbb{H}}^{n}$ by right multiplication and use the
convention that $\text{SO}(2n,2n)$ acts on
$\widetilde{\mathbb{H}}^{n}$ on the left.

An \emph{almost product structure} on a smooth manifold $M$ is a
tensor field $P$ of type (1,1) on $M$, $P\neq\pm Id$, such that
$P^2=Id,$ where $Id$ is the identity tensor field of type $(1,1)$ on
$M$. The pair $(M,P)$ is called an \emph{almost product manifold}.
An \emph{almost para-complex manifold} is an almost product manifold
$(M,P)$ such that the two eigenbundles $T^+M$ and $T^-M$ associated
with the two eigenvalues $+1$ and $-1$ of $P$, respectively, have the
same rank. Equivalently, a splitting of the tangent bundle $TM$ of a
differentiable manifold $M$, into the Whitney sum of two subbundles
$T^±M$ of the same fiber dimension is called an \emph{almost
para-complex structure} on M.

An \emph{almost para-Hermitian structure} on a smooth manifold $M$
is a pair $(g,P)$, where $g$ is a pseudo-Riemannian metric on $M$
and $P$ is an almost product structure on $M$, which is compatible
with $g$, i.e. $P^{*}g =-g$. In this case, the triple $(M,g,P)$ is
called an \emph{almost para-Hermitian manifold}. Moreover, $(M,g,P)$
is said to be a \emph{para-Hermitian manifold} if $P$ is integrable,
i.e. if the Nijenhuis $N_P$ defined by
\[
         N_P(X,Y)=[PX, PY]-P[X,PY]
         -P[PX,Y]+[X,Y]
\]
vanishes. An \emph{almost complex structure} on a smooth manifold
$M$ is a tensor field $J$ of type $(1,1)$ on $M$ such that
$J^2=-Id$. The pair $(M,J)$ is called an \emph{almost complex
manifold}. We note that the dimension of an almost (para-)complex
manifold is necessarily even (see \cite{CFG,KN}).

An \emph{almost pseudo-Hermitian structure} on a smooth manifold $M$
is a pair $(g,J)$, where $g$ is a pseudo-Riemannian metric on $M$
and $J$ is an almost complex structure on $M$, which is compatible
with $g$, i.e. $J^{*}g =g$. In this case, the triple $(M,g,J)$ is
called an \emph{almost pseudo-Hermitian manifold}. Moreover,
$(M,g,J)$ is said to be \emph{pseudo-Hermitian} if $J$ is
integrable, i.e. if the Nijenhuis $N_J$ defined by
\[
         N_J(X,Y)=[JX, JY]-J[X,JY]
         -J[JX,Y]-[X,Y]
\]
vanishes.

An \emph{almost para-hypercomplex structure} on a smooth manifold
$M$ is a triple $H=(J_1,J_2,J_3)$ of $(1,1)$–-type tensor fields on
$M$ satisfying:
\begin{equation}\label{V25}
        J_\alpha^2=-\tau_\alpha {\rm Id},\ J_\alpha J_\beta=-J_\beta J_\alpha=\tau_\gamma J_\gamma,
\end{equation}
for any $\alpha\in\{1,2,3\}$ and for any even permutation
$(\alpha,\beta,\gamma)$ of $(1,2,3)$, where $\tau_1= \tau_2=-1=-\tau_3$.
In this case $(M,H)$ is said to be an \emph{almost para-hypercomplex
manifold}. We remark that from (\ref{V25}) it follows that $J_1$ and
$J_2$ are almost product structures on $M$, while $J_3=J_1J_2$ is an
almost complex structure on $M$. We note that the almost
para-hypercomplex structures have been introduced in geometry by P.
Libermann \cite{L} under the name of quaternionic structures of
second kind (\emph{structures presque quaternioniennes de
deuxi\`{e}me esp\`{e}ce}).

A semi-Riemannian metric $g$ on $(M,H)$ is said to be
\emph{compatible} or \emph{adapted} to the almost para-hypercomplex
structure $H=(J_{\alpha})_{\alpha=1,2,3}$ if it satisfies:
\[
         g(J_\alpha X,J_\alpha Y)=\tau_{\alpha}
         g(X,Y)
\]
for all vector fields $X$,$Y$ on $M$ and $\alpha\in\{1,2,3\}$.
Moreover, the pair $(g,H)$ is called an \emph{almost
para-hyperhermitian structure} on $M$ and the triple $(M,g,H)$ is
said to be an \emph{almost para-hyperhermitian manifold}. We note
that any almost para-hyperhermitian manifold is of dimension $4m,\
m\geq 1$, and any adapted metric is necessarily of neutral signature
$(2m,2m)$. If $\lbrace{J_1,J_2,J_3}\rbrace$ are parallel with respect
to the Levi-Civita connection of $g$, then the manifold is called
\emph{para-hyper-K\"{a}hler}.

An almost para-hypercomplex manifold $(M,H)$ is called a
\emph{para-hypercomplex manifold} if each $J_{\alpha}$,
$\alpha=1,2,3$, is integrable. In this case $H$ is said to be a
\emph{para-hypercomplex structure} on $M$. Moreover, if $g$ is a
semi-Riemannian metric adapted to the para-hypercomplex structure
$H$, then the pair $(g,H)$ is said to be a \emph{para-hyperhermitian
structure} on $M$ and $(M,g,H)$ is called a
\emph{para-hyperhermitian manifold}.

An \emph{almost paraquaternionic Hermitian manifold} is a triple
$(M,\sigma,g)$, where $M$ is a smooth manifold, $\sigma$ is an
almost paraquaternionic structure on $M$, \emph{i.e.} a rank
3-subbundle of $End(TM)$ which is locally spanned by an almost
para-hypercomplex structure $H=(J_{\alpha})_{\alpha=1,2,3}$ and $g$
is a compatible metric with respect to $H$. We remark that, if
$\{J_1,J_2,J_3\}$ and $\{J'_1,J'_2,J'_3\}$ are two canonical local
bases of $\sigma$ in $U$ and in another coordinate neighborhood $U'$
of $M$, then we have for all $x\in U\cap U'$
\begin{equation}\label{V14}
\left(J'_\alpha\right)_x=\sum_{\beta=1}^{3}s_{\alpha\beta}(x)\left(J'_\beta\right)_x,\
\alpha=1,2,3,
\end{equation}
where
$S_{UU'}(x)=\left(s_{\alpha\beta}(x)\right)_{\alpha,\beta=1,2,3}\in
\text{SO}(2,1)$, because $\{J_1,J_2,J_3\}$ and $\{J'_1,J'_2,J'_3\}$
satisfy the paraquaternionic identities (\ref{V25}).

If $(M,\sigma,g)$ is an almost paraquaternionic Hermitian manifold
such that the bundle $\sigma$ is preserved by the Levi-Civita
connection $\nabla$ of $g$, then $(M,\sigma,g)$ is said to be a
\emph{paraquaternionic K\"{a}hler manifold} \cite{GRM}.
Equivalently, we can write
\begin{equation}\label{pK}
\nabla J_\alpha=\tau_\beta\omega_\gamma\otimes
J_\beta-\tau_\gamma\omega_\beta\otimes J_\gamma,
\end{equation}
where $(\alpha,\beta,\gamma)$ is an even permutation of $(1,2,3)$
and $\omega_1,\omega_2,\omega_3$ are locally defined 1-forms. We
note that the prototype of paraquaternionic K\"{a}hler manifold is
the paraquaternionic projective space $P^n(\widetilde{\mathbb{H}})$
as described by Bla\v{z}i\'{c} \cite{BLZ}.

If the Riemannian curvature tensor $R$ is taken with the sign
convention
\[R(X,Y)Z = \nabla_{X}\nabla_{Y}Z- \nabla_{Y}\nabla_{X}Z -
\nabla_{[X,Y]}Z,\]  for all vector fields $X,Y,Z$ on $M$, then a
consequence of (\ref{pK}) is that $R$ satisfies
\begin{equation}\label{rj}
[R, J_\alpha]=\tau_\beta A_\gamma\otimes J_\beta-\tau_\gamma
A_\beta\otimes J_\gamma,
\end{equation}
for any even permutation $(\alpha,\beta,\gamma)$ of $(1,2,3)$, where
\[
A_\alpha=d\omega_\alpha+\tau_\alpha\omega_\beta\wedge \omega_\gamma.
\]

If we consider $\Omega_\alpha:=g(J_\alpha\cdot,\cdot)$ the
fundamental form associated with $J_\alpha$, $\alpha=1,2,3$, then we
have the following structure equations (see \cite{AC,BDM}):
\begin{equation}\label{V13}
d\omega_\alpha+\tau_\alpha\omega_\beta\wedge
\omega_\gamma=\tau_\alpha\nu\Omega_\alpha,
\end{equation}
for any even permutation $(\alpha,\beta,\gamma)$ of $(1,2,3)$, where
$\nu=\frac{Sc}{4n(n + 2)}$ is the reduced scalar curvature, $Sc$
being the scalar curvature defined as the trace of the Ricci tensor
$\rho$.

We recall that the main property of manifolds endowed with this kind
of structure is the following.

\begin{thm}\label{2.1} \cite{GRM}
Any paraquaternionic K\"{a}hler manifold $(M,\sigma,g)$ is an
Einstein space, provided that ${\rm dim} M >4$.
\end{thm}

Let $(M,\sigma,g)$ be a $4n$-dimensional paraquaternionic K\"{a}hler
manifold. Then the Ricci 2-forms of the Levi-Civita
connection of $g$ are defined as (see \cite{IMZ}):
\[\rho_{\alpha}(X,Y)=- \frac{\tau_\alpha}{2}\text{Trace}(Z\rightarrow J_{\alpha}R(X,Y)Z),
\ \alpha=1,2,3,\] and for $n>1$ it follows
\begin{equation}\label{ro}
      \rho(X,Y)=\frac{n+2}{n}\rho_{\alpha}(X,J_{\alpha}Y),\
      \alpha=1,2,3.
\end{equation}
Using now Theorem \ref{2.1} we obtain the following relations
\begin{eqnarray}\label{roi}
      &&\rho_{\alpha}(X,Y)=-\tau_\alpha\frac{Sc}{4(n+2)}g(X,J_{\alpha}Y),\
      \alpha=1,2,3.
\end{eqnarray}

On the other hand, from (\ref{rj}) and (\ref{ro}) and taking account
of Theorem \ref{2.1}, we obtain that the curvature tensor of a
$4n$-dimensional paraquaternionic K\"{a}hler manifold ($n>1$),
satisfies (see \cite{GRM}):
\begin{eqnarray}\label{ror}
      &&R(X,Y,J_{1}Z,J_{1}T)+R(X,Y,Z,T)\nonumber\\
       &&=\frac{Sc}{4n(n+2)}\{g(X,J_{3}Y)g(Z,J_{3}T)-g(X,J_{2}Y)g(Z,J_{2}T)\}\nonumber\\
      &&R(X,Y,J_{2}Z,J_{2}T)+R(X,Y,Z,T)\nonumber\\
      &&=\frac{Sc}{4n(n+2)}\{g(X,J_{3}Y)g(Z,J_{3}T)-g(X,J_{1}Y)g(Z,J_{1}T)\}\\
      &&R(X,Y,J_{3}Z,J_{3}T)-R(X,Y,Z,T)\nonumber\\
      &&=\frac{Sc}{4n(n+2)}\{g(X,J_{2}Y)g(Z,J_{2}T)+g(X,J_{1}Y)g(Z,J_{1}T)\},\nonumber
\end{eqnarray}
for all vector fields $X,Y,Z$ and $T$ on $M$, where
$R(X,Y,Z,T)=g(R(X,Y)Z,T)$.

We consider now the general case of a $4n$-dimensional smooth
manifold $M$ endowed with an almost paraquaternionic structure
$\sigma$ and with a \emph{paraquaternionic connection} $\nabla$,
\emph{i.e.} a linear connection which preserves $\sigma$, and
following \cite{IMZ} we recall some basic facts concerning the
twistor and reflector spaces of $M$, which will be useful in the
next Section.

Let $p\in M$. Any linear frame $u$ on $T_{p}M$ can be considered as
an isomorphism $u:\mathbb{R}^{4n}\rightarrow T_{p}M$. Taking such a
frame $u$ we can define a subspace of of the space of the all
endomorphisms of $T_{p}M$ by  $u(\text{sp}(1,\mathbb{R}))u^{-1}$.
This subset is a paraquaternionic structure and we define $P(M)$ to
be the set of all linear frames $u$ which satisfy
$u(\text{sp}(1,\mathbb{R}))u^{-1}=\sigma$, where
$\text{sp}(1,\mathbb{R})=\text{Span}
\{J^{0}_{1},J^{0}_{2},J^{0}_{3}\}$ is the Lie algebra of
$\text{Sp}(1,\mathbb{R})$. It is clear that $P(M)$ is the principal
frame bundle of $M$ with structure group
$\text{GL}(n,\widetilde{\mathbb{H}})\text{Sp}(1,\mathbb{R})$, where
\[\text{GL}(n,\widetilde{\mathbb{H}})=\{A\in
\text{GL}(4n,\mathbb{R})|A(\text{sp}(1,\mathbb{R}))A^{-1}=\text{sp}(1,\mathbb{R})\}.\]

We denote by $\pi:P(M)\rightarrow M$ the natural projection and
remark that the Lie algebra of $\text{GL}(n,\widetilde{\mathbb{H}})$
is
\[\text{gl}(n,\widetilde{\mathbb{H}})=\{A\in \text{gl}(4n,\mathbb{R})|AB=BA\
\text{for all } B\in \text{sp}(1,\mathbb{R})\}.\]

We also denote by $(\ ,\ )$ the inner product in
$\text{gl}(4n,\mathbb{R})$ given by $(A,B) = \text{Trace}(AB^{t})$,
for  $A, B \in \text{gl}(4n, \mathbb{R}).$

We split now the curvature of $\nabla$ into
$\text{gl}(n,\widetilde{\mathbb{H}})$-valued part $R'$ and
$\text{sp}(1,\mathbb{R})$-valued part $R''$ following the classical scheme
(see e.g. \cite{Be}). We denote the splitting of the
$\text{gl}(n,\widetilde{\mathbb{H}})\oplus \text{sp}(1,\mathbb{R})$-valued
curvature 2-form $\Omega$ on $P(M)$ according to the splitting of
the curvature $R$, by
\begin{equation}\label{om1}
\Omega= \Omega'+ \Omega'',
\end{equation}
where $\Omega'$ is a $\text{gl}(n,\widetilde{\mathbb{H}})$-valued
2-form and $\Omega''$ is a $\text{sp}(1,\mathbb{R})$-valued form.
Explicitly,
\begin{equation}\label{om2}
\Omega''=\Omega''_{1}J^{0}_{1}+\Omega''_{2}J^{0}_{2}+\Omega''_{3}J^{0}_{3},
\end{equation}
where $\Omega''_{\alpha},\ \alpha=1,2,3$ are 2-forms. If $\xi, \eta
\in \mathbb{R}^{4n}$, then the 2-forms $\Omega''_{\alpha},\
\alpha=1,2,3$, are given by

\begin{equation}\label{omei}
      \Omega''_{\alpha}(B(\xi),B(\eta))=\frac{1}{2n}\rho_{\alpha}(X,Y)=-\tau_\alpha\frac{Sc}{8n(n+2)}g(X,J_{\alpha}Y),
\end{equation}
where $X=u(\xi),\ Y=u(\eta)$ (see \cite{IMZ}).

%\section{Twistor and reflector spaces of a paraquaternionic K\"{a}hler manifold}

%In this section, following \cite{AC,IMZ,BDM}, we recall some basic facts concerning twistor and reflector spaces in  paraquaternionic setting, which will be useful in Section 4.

%Let $(M,\sigma,g)$ be a $4n$-dimensional paraquaternionic K\"{a}hler manifold, $P(M)$ be the principal frame bundle of $M$ and let $\pi:P(M)\rightarrow M$ be the natural projection.

For each $u \in P(M)$ we consider two linear isomorphism $j^{+}(u)$ and $j^{-}(u)$
on $T_{\pi(u)}M$ defined by:
\[j^{+}(u) = uJ^{0}_{1}u^{-1},\ \
j^{-}(u) = uJ^{0}_{3}u^{-1}.\]

It is easy to see that
\[(j^{-}(u))^{2} = -Id ,\ \ (j^{+}(u))^{2} = Id\]
and
\[g(j^{-}(u)X,j^{-}(u)Y) =
g(X,Y),\ \ g(j^{+}(u)X,j^{+}(u)Y) = -g(X,Y),\] for all $X,Y \in
T_{\pi(u)}M$.

As in \cite{IMZ}, for each $p\in M$ we consider
\[Z^{\pm}_{p}(M)=\{j^{\pm}(u)| u \in P(M), \pi(u)=p \}\]
and we define the \emph{twistor space} $Z^{-}$ of $M$ and the
\emph{reflector space} $Z^{+}$ of $M$, by setting
\[Z^{\pm}=Z^{\pm}(M)=\bigcup _{p\in M}Z^{\pm}_{p}(M).\]

Then the twistor space $Z^{-}(M)$ is the unit pseudo-sphere bundle with fibre the 2-sheeted hyperboloid
$S^{2}_{1}(-1)=\{(x,y,z)\in\mathbb{R}^{3}|x^{2}+y^{2}-z^{2}=-1\}$ and the reflector space $Z^{+}(M)$ is the unit pseudo-sphere bundle with fibre the 1-sheeted hyperboloid $S^{2}_{1}(1)=\{(x,y,z)\in\mathbb{R}^{3}|x^{2}+y^{2}-z^{2}=1\}$.

We denote by $A^{*}$ (resp. $B(\xi)$) the fundamental vector field
(resp. the standard horizontal vector field) on $P(M)$, the
principal frame bundle of $M$, corresponding to $A\in
\text{gl}(n,\widetilde{\mathbb{H}})\oplus \text{sp}(1,\mathbb{R})$ (resp.
$\xi\in\mathbb{R}^{4n}$). Let $u \in P(M)$ and $Q_u$ be the
horizontal subspace of the tangent space $T_{u}P(M)$ induced by the
connection $\nabla$ on $M$ (see \cite{KN}). As in \cite{IMZ} we have
the decompositions
\[T_{u}P(M) = (h_{i})^{*}_{u}\oplus (m_{i})^{*}_{u}\oplus Q_u,\
i=1,3\] and the following isomorphisms
\[j^{-}_{\ast u}|_{(m_{3})^{*}_{u}\oplus Q_u}:(m_{3})^{*}_{u}\oplus Q_u \rightarrow T_{j^{-}(u)}Z^{-},\]
\[j^{+}_{\ast u}|_{(m_{1})^{*}_{u}\oplus Q_u}:(m_{1})^{*}_{u}\oplus Q_u \rightarrow T_{j^{+}(u)}Z^{+}\]
where \[h_{3} = \{A\in \text{gl}(n,\widetilde{\mathbb{H}})\oplus
\text{sp}(1,\mathbb{R})| AJ^{0}_{3} = J^{0}_{3}A\},\] \[h_{1} = \{A\in
\text{gl}(n,\widetilde{\mathbb{H}})\oplus \text{sp}(1,\mathbb{R})| AJ^{0}_{1} =
J^{0}_{1}A\},\]
\begin{equation}\label{m3}
m_{3} = \{A \in \text{gl}(n,\widetilde{\mathbb{H}})\oplus
\text{sp}(1,\mathbb{R})| AJ^{0}_{3} = -J^{0}_{3}A\}=\text{Span}\{J^{0}_{1}, J^{0}_{2}\},
\end{equation}
\begin{equation}\label{m1}
m_{1}= \{A \in \text{gl}(n,\widetilde{\mathbb{H}})\oplus
\text{sp}(1,\mathbb{R})| AJ^{0}_{1} = -J^{0}_{1}A\}=\text{Span}\{J^{0}_{2}, J^{0}_{3}\},
\end{equation}
and \[(h_{i})^{*}_{u}= \{A^{*}_{u}| A\in h_{i}\},\ (m_{i})^{*}_{u}=
\{A^{*}_{u}| A\in m_{i}\},\ i=1,3.\]

Now, we can define two almost complex structures $I_1$ and $I_2$
on $Z^{-}$ by (see \cite{IMZ})
\begin{eqnarray}\label{ti}
       &&I_1(j^{-}_{\ast u}A^{*})=j^{-}_{\ast u}(J^{0}_{3}A)^{*},\nonumber\\
       &&I_2(j^{-}_{\ast u}A^{*})=-j^{-}_{\ast u}(J^{0}_{3}A)^{*}\\
       &&I_i(j^{-}_{\ast u}B(\xi))=j^{-}_{\ast u}B(J^{0}_{3}\xi),\ i=1,2,\nonumber
\end{eqnarray}
for $u \in P(M),\ A\in m_{3},\ \xi\in \mathbb{R}^{4n}$.

Similarly, it can be defined two almost paracomplex structures $P_1$
and $P_2$ on $Z^{+}$ by (see \cite{IMZ})
\begin{eqnarray}\label{tp}
       &&P_1(j^{+}_{\ast u}A^{*})=j^{+}_{\ast u}(J^{0}_{1}A)^{*},\nonumber\\
       &&P_2(j^{+}_{\ast u}A^{*})=-j^{+}_{\ast u}(J^{0}_{1}A)^{*}\\
       &&P_i(j^{+}_{\ast u}B(\xi))=j^{+}_{\ast u}B(J^{0}_{1}\xi),\ i=1,2,\nonumber
\end{eqnarray}
for $u \in P(M),\ A\in m_{1},\ \xi\in \mathbb{R}^{4n}$.

We remark that the almost complex structures defined above  were also defined and investigated in \cite{BDM} on paraquaternionic K\"{a}hler manifolds, the authors proving that the
almost complex structure $I_2$ is never integrable while $I_1$ is
always integrable. Moreover, we note that in \cite{IMZ} the authors
found that the para-complex structures $P_2$ is never integrable on
reflector space, while $P_1$ is always integrable.

\section{Curvature properties of $(Z^{-}, I_i, h_t)$ and $(Z^{+}, P_i, h_t)$, $i = 1, 2$ on paraquaternionic K\"{a}hler
manifolds}

Let $(M,\sigma,g)$ be a $4n$-dimensional paraquaternionic K\"{a}hler
manifold. Since the Levi-Civita connection $\nabla$ of $g$ is a
paraquaternionic connection on $M$, it follows that all the
considerations in the last part of the above sections remain valid
if the manifold is endowed with a paraquaternionic K\"{a}hler
structure. Then, as in \cite{AC,BDM,IZ}, we can define a natural
1-parameter family of pseudo-Riemannian metrics $h_{t}$, $t\neq 0$,
on $Z^{-}$  by
\begin{eqnarray}\label{m-}
&h_{t}(j^{-}_{\ast u}A^{*},j^{-}_{\ast u}B^{*})=t(A, B),\nonumber\\
&h_{t}(j^{-}_{\ast u}A^{*},j^{-}_{\ast u}B(\xi))=0,\\
&h_{t}(j^{-}_{\ast u}B(\xi),j^{-}_{\ast u}B(\eta))=g(X,Y),\nonumber
\end{eqnarray}
for $u \in P(M),\ A,B\in m_{3},\ \xi,\eta\in \mathbb{R}^{4n}$ and
$X=u(\xi),\ Y=u(\eta)$, and similarly on $Z^{+}$:
\begin{eqnarray}\label{m+}
&h_{t}(j^{+}_{\ast u}A^{*},j^{+}_{\ast u}B^{*})=t(A, B),\nonumber\\
&h_{t}(j^{+}_{\ast u}A^{*},j^{+}_{\ast u}B(\xi))=0,\\
&h_{t}(j^{+}_{\ast u}B(\xi),j^{+}_{\ast u}B(\eta))=g(X,Y),\nonumber
\end{eqnarray}
for $u \in P(M),\ A,B\in m_{1},\ \xi,\eta\in \mathbb{R}^{4n}$ and
$X=u(\xi),\ Y=u(\eta)$.

It is easy now to see that $(I_i, h_t)$, $i = 1, 2$ (resp. $(P_i,
h_t)$, $i = 1, 2$) determine two families of almost pseudo-Hermitian
structures on $Z^{-}$ (resp. two families of almost para-Hermitian
structures on $Z^{+}$).

The computations performed in \cite{Sk} for the metric twistor space
over Riemannian manifold $M$ can be adapted in our cases, using
(\ref{ti}) and (\ref{m-}) for the twistor space and respectively
(\ref{tp}) and (\ref{m+}) for the reflector space. We obtain for the
curvature tensors $K^{\pm}$, where $K^{-}$ is the curvature tensor on
the twistor space and $K^{+}$ is the curvature tensor on the
reflector space, the following expressions:
\begin{eqnarray}\label{k}
&K^{\pm}(j^{\pm}_{\ast u}A^{*},j^{\pm}_{\ast u}B^{*},j^{\pm}_{\ast u}C^{*},j^{\pm}_{\ast u}D^{*})=-t([A, B],[C, D])\nonumber\\
&K^{\pm}(j^{\pm}_{\ast u}A^{*},j^{\pm}_{\ast u}B^{*},j^{\pm}_{\ast u}C^{*},j^{\pm}_{\ast u}B(\xi))=0\nonumber\\
&K^{\pm}(j^{\pm}_{\ast u}A^{*},j^{\pm}_{\ast u}B(\xi),j^{\pm}_{\ast u}B^{*},j^{\pm}_{\ast u}B(\eta))
=\frac{t}{2}([A,B],\Omega(B(\xi),B(\eta))(u))-\nonumber\\
&-\frac{t^{2}}{4}\varepsilon_{i}(B,\Omega(B(\xi),B(e_i))(u))(A,\Omega(B(\eta),B(e_i))(u))\\
&K^{\pm}(j^{\pm}_{\ast u}B(\xi),j^{\pm}_{\ast
u}B(\eta),j^{\pm}_{\ast u}B(\zeta),j^{\pm}_{\ast u}A^{*})
=\frac{t}{2}(A,B(\zeta)_{u}(\Omega(B(\xi),B(\eta))))\nonumber\\
&K^{\pm}(j^{\pm}_{\ast u}B(\xi),j^{\pm}_{\ast
u}B(\eta),j^{\pm}_{\ast u}B(\zeta),j^{\pm}_{\ast u}B(\tau))=
R(u(\xi),u(\eta),u(\zeta),u(\tau))-\nonumber\\
&-\frac{t}{4}\{(\Omega_{m_{\pm}}(B(\eta),B(\zeta))(u),\Omega_{m_{\pm}}(B(\xi),B(\tau))(u))-\nonumber\\
&-(\Omega_{m_{\pm}}(B(\xi),B(\zeta))(u),\Omega_{m_{\pm}}(B(\eta),B(\tau))(u))-\nonumber\\
&-2(\Omega_{m_{\pm}}(B(\xi),B(\eta))(u),\Omega_{m_{\pm}}(B(\zeta),B(\tau))(u))\}\nonumber
\end{eqnarray}
for $u \in P(M),\ A,B,C,D\in m_{\pm}$, where $m_{-}=m_{3}$,
$m_{+}=m_{1}$, $\xi,\eta,\zeta,\tau\in \mathbb{R}^{4n}$ and
$\Omega_{m_{\pm}}$ denotes the $m_{\pm}$-component of $\Omega$ .

Using now (\ref{om1})--(\ref{m1})  in (\ref{k}) we derive
\begin{eqnarray}\label{kp-}
&K^{-}(j^{-}_{\ast u}A^{*},j^{-}_{\ast u}B^{*},j^{-}_{\ast u}C^{*},j^{-}_{\ast u}D^{*})
=\frac{t}{n}(J^{0}_{3}A, B)(J^{0}_{3}C, D)\nonumber\\
&K^{-}(j^{-}_{\ast u}A^{*},j^{-}_{\ast u}B^{*},j^{-}_{\ast u}C^{*},j^{-}_{\ast u}B(\xi))=0\nonumber\\
&K^{-}(j^{-}_{\ast u}A^{*},j^{-}_{\ast u}B(\xi),j^{-}_{\ast u}B^{*},j^{-}_{\ast u}B(\eta))= \nonumber\\
&=(-\frac{t^{2}Sc^{2}}{64n(n+2)^{2}}+\frac{t Sc}{8n(n+2)})(J^{0}_{3}A,
B)g(J_{3}X,Y)
+\frac{t^{2}Sc^{2}}{64n(n+2)^{2}}(A,B)g(X,Y)\nonumber\\
&K^{-}(j^{-}_{\ast u}B(\xi),j^{-}_{\ast u}B(\eta),j^{-}_{\ast u}B(\zeta),j^{-}_{\ast u}A^{*})=0\\
&K^{-}(j^{-}_{\ast u}B(\xi),j^{-}_{\ast u}B(\eta),j^{-}_{\ast u}B(\zeta),j^{-}_{\ast u}B(\tau))=R(X,Y,Z,T)-\nonumber\\
&-\frac{tSc^{2}}{64n(n+2)^{2}}\{g(Y,J_{1}Z)g(X,J_{1}T)+g(Y,J_{2}Z)g(X,J_{2}T)-\nonumber\\
&-g(X,J_{1}Z)g(Y,J_{1}T)-g(X,J_{2}Z)g(Y,J_{2}T)-\nonumber\\
&-2g(X,J_{1}Y)g(Z,J_{1}T)-2g(X,J_{2}Y)g(Z,J_{2}T)\}\nonumber
\end{eqnarray}
for $u \in P(M),\ A,B,C,D\in m_3,\ \xi,\eta,\zeta,\tau\in
\mathbb{R}^{4n}$ and $u(\xi)=X,\ u(\eta)=Y,\ u(\zeta)=Z,\ u(\tau)=T$
 and
\begin{eqnarray}\label{kp+}
&K^{+}(j^{+}_{\ast u}A^{*},j^{+}_{\ast u}B^{*},j^{+}_{\ast u}C^{*},j^{+}_{\ast u}D^{*})
=\frac{t}{n}(J^{0}_{1}A, B)(J^{0}_{1}C, D)\nonumber\\
&K^{+}(j^{+}_{\ast u}A^{*},j^{+}_{\ast u}B^{*},j^{+}_{\ast u}C^{*},j^{+}_{\ast u}B(\xi))=0\nonumber\\
&K^{+}(j^{+}_{\ast u}A^{*},j^{+}_{\ast u}B(\xi),j^{+}_{\ast u}B^{*},j^{+}_{\ast u}B(\eta))=\nonumber\\
&=(-\frac{t^{2}Sc^{2}}{64n(n+2)^{2}}-\frac{t Sc}{8n(n+2)})(J^{0}_{1}A,
B)g(J_{1}X,Y)-\frac{t^{2}Sc^{2}}{64n(n+2)^{2}}(A,B)g(X,Y)\\
&K^{+}(j^{+}_{\ast u}B(\xi),j^{+}_{\ast u}B(\eta),j^{+}_{\ast u}B(\zeta),j^{+}_{\ast u}A^{*})=0\nonumber\\
&K^{+}(j^{+}_{\ast u}B(\xi),j^{+}_{\ast u}B(\eta),j^{+}_{\ast u}B(\zeta),j^{+}_{\ast u}B(\tau))=R(X,Y,Z,T)-\nonumber\\
&-\frac{t Sc^{2}}{64n(n+2)^{2}}\{-g(Y,J_{2}Z)g(X,J_{2}T)+g(Y,J_{3}Z)g(X,J_{3}T)+\nonumber\\
&+g(X,J_{2}Z)g(Y,J_{2}T)-g(X,J_{3}Z)g(Y,J_{3}T)+\nonumber\\
&+2g(X,J_{2}Y)g(Z,J_{2}T)-2g(X,J_{3}Y)g(Z,J_{3}T)\}\nonumber
\end{eqnarray}
for $u \in P(M),\ A,B,C,D\in m_1,\ \xi,\eta,\zeta,\tau\in
\mathbb{R}^{4n}$ and $u(\xi)=X,\ u(\eta)=Y,\ u(\zeta)=Z,\
u(\tau)=T$.

We recall now that the *-Ricci tensor of a $2n$-dimensional almost
pseudo-Hermitian manifold $(M,g,J)$ is defined by \[\rho^*(X,
Y)=\sum_{i=1}^{2n}\varepsilon_{i}R(X,E_{i}, JY, JE_{i}),\] where $R$
denotes the curvature of the metric $g$, $\{E_{1}, ..., E_{2n}\}$ is
a pseudo-orthonormal basis at an arbitrary point $p$ and $X, Y$ are
tangent vectors at $ p$. If the *-Ricci tensor is scalar multiple of
the metric then the manifold is said to be \emph{*-Einstein}.

On the other hand, we note that A. Gray introduced in \cite{G} three
basic classes $AH1,\ AH2,\ AH3$ of almost Hermitian manifolds, whose
curvature tensors resemble that of a K\"{a}hler manifold. They are
defined by the following curvature identities:
\[AH1: R(X,Y,Z,T)=R(X,Y,JZ,JT),\]
\[AH2: R(X,Y,Z,T)=R(JX,JY,Z,T)+R(JX,Y,JZ,T)+R(JX,Y,Z,JT),\]
\[AH3: R(X,Y,Z,T)=R(JX, JY, JZ, JT),\]
where $R$ is the curvature tensor of the manifold. It is easy to see
that \[AH1 \subset AH2 \subset AH3.\]

By analogy, one says that an almost para-Hermitian manifold
satisfies the para-Gray identities if
\[APH1: R(X,Y,Z,T)=-R(X,Y,JZ,JT),\]
\[APH2: R(X,Y,Z,T)=-R(JX,JY,Z,T)-R(JX,Y,JZ,T)-R(JX,Y,Z,JT),\]
\[APH3: R(X,Y,Z,T)=R(JX, JY, JZ, JT),\]
where $R$ is the curvature tensor of the manifold. We note that
para-Gray-like identities were considered in \cite{BV,CL}) and it
can be easily checked that \[APH1 \subset APH2 \subset APH3.\]

We can give now the main curvature properties of  $(Z^{-}, I_{i},
h_{t})$ and $(Z^{+}, P_{i}, h_{t})$, $i = 1, 2$, in the following
two theorems.

\begin{thm}\label{TT}
Let $(M,\sigma,g)$ be a $4n$-dimensional paraquaternionic K\"{a}hler
manifold and $(Z^{-}, I_{i}, h_{t})$, $i=1,2$  the twistor spaces
associated. Then:
\\$(i)$ The manifolds $(Z^{-}, I_{i}, h_{t}),\
i=1,2$, belong always to $AH2$ and $AH3$ and are with
pseudo-Hermitian Ricci tensor and with pseudo-Hermitian *-Ricci
tensor;
\\$(ii)$ The manifolds  $(Z^{-}, I_{1}, h_{t})$ belong to $AH1$ iff $Sc = 0$ or $Sc =
\frac{4(n + 2)}{t}$;
\\$(iii)$ The manifolds  $(Z^{-}, I_{2}, h_{t})$ belong to $AH1$ iff $Sc = 0$;
\\$(iv)$ The manifolds $(Z^{-}, I_{i}, h_{t}), i=1,2$  are Einstein iff
\[ Sc = \frac{ 4(n + 2) }{ t }\ \text{\rm or}\ Sc = \frac{4(n+2)}{ (n+1)t
};\]
\\$(v)$ The manifolds $(Z^{-}, I_{1}, h_{t})$ are *-Einstein iff
\[ Sc = \frac{4(n + 2)}{ t}\ \text{\rm or}\ Sc = -\frac{4(n+2)}{ nt};\]
\\$(vi)$ The manifolds $(Z^{-}, I_{2}, h_{t})$ are *-Einstein iff
\[ Sc=\frac{2(n+2)}{ (n-1)t }[3n-1-\sqrt{ 9n^{2}-10n+5 }] \]
or
\[ Sc=\frac{2(n+2)}{ (n-1)t }[3n-1+\sqrt{ 9n^{2}-10n+5 }]. \]
\end{thm}

\begin{proof}
$(i)$  This statement follows from  (\ref{p1}), (\ref{ror}), (\ref{m3}),  (\ref{ti}) and (\ref{kp-}) after some long
but straightforward computations.\\
$(ii)+(iii)$ From (\ref{ti}) and (\ref{kp-})  we obtain

\begin{eqnarray}\label{-ah1}
&K^{-}(j^{-}_{\ast u}A^{*},j^{-}_{\ast u}B^{*},I_i(j^{-}_{\ast u}C^{*}),I_i(j^{-}_{\ast u}D^{*}))
=\frac{t}{n}(J^{0}_{3}A, B)(J^{0}_{3}C, D)\nonumber\\
&K^{-}(j^{-}_{\ast u}A^{*},j^{-}_{\ast u}B^{*},I_i(j^{-}_{\ast u}C^{*}),I_i(j^{-}_{\ast u}B(\xi)))=0\nonumber\\
&K^{-}(j^{-}_{\ast u}A^{*},j^{-}_{\ast u}B(\xi),I_i(j^{-}_{\ast u}B^{*}),I_i(j^{-}_{\ast u}B(\eta)))= \nonumber\\
&=(\varepsilon_i \frac{t^{2}Sc^{2}}{64n(n+2)^{2}}-\varepsilon_i\frac{t Sc}{8n(n+2)})(A,B)g(X,Y)
-\varepsilon_i\frac{t^2 Sc^2}{64n(n+2)^2}(J^{0}_{3}A,
B)g(J_{3}X,Y)\nonumber\\
&K^{-}(j^{-}_{\ast u}B(\xi),j^{-}_{\ast u}B(\eta),I_i(j^{-}_{\ast u}B(\zeta)),I_i(j^{-}_{\ast u}A^{*}))=0\\
&K^{-}(j^{-}_{\ast u}B(\xi),j^{-}_{\ast u}B(\eta),I_i(j^{-}_{\ast u}B(\zeta)),I_i(j^{-}_{\ast u}B(\tau)))=R(X,Y,Z,T)-\nonumber\\
&-\frac{tSc^{2}}{64n(n+2)^{2}}\{g(Y,J_{1}Z)g(X,J_{1}T)+g(Y,J_{2}Z)g(X,J_{2}T)-\nonumber\\
&-g(X,J_{1}Z)g(Y,J_{1}T)-g(X,J_{2}Z)g(Y,J_{2}T)\}+\nonumber\\
&+(\frac{Sc}{4n(n+2)}-\frac{2tSc^{2}}{64n(n+2)^{2}})(g(X,J_{1}Y)g(Z,J_{1}T)+g(X,J_{2}Y)g(Z,J_{2}T))\nonumber
\end{eqnarray}
for $i=1,2, \varepsilon_1=-1,\ \varepsilon_2=+1,\ \ u \in P(M),\
A,B,C,D\in m_3,\ \xi,\eta,\zeta,\tau\in \mathbb{R}^{4n}$ and
$u(\xi)=X,\ u(\eta)=Y,\ u(\zeta)=Z,\ u(\tau)=T.$ Now the conclusion
follows from (\ref{kp-}) and (\ref{-ah1}).
\\
$(iv)+(v)+(vi)$ Using (\ref{kp-})  we obtain the
following formulas for the Ricci tensors $\rho^{-}$ of $(Z^{-},
h_{t})$:

\begin{eqnarray}\label{ro-}
&\rho^{-}(j^{-}_{\ast u}A^{*},j^{-}_{\ast u}B^{*})=
[\frac{tSc^{2}}{16(n+2)^{2}}+\frac{1}{nt}]h_{t}(j^{-}_{\ast u}A^{*},j^{-}_{\ast u}B^{*})\nonumber\\
&\rho^{-}(j^{-}_{\ast u}A^{*},j^{-}_{\ast u}B(\xi))=0\\
&\rho^{-}(j^{-}_{\ast u}B(\xi),j^{-}_{\ast
u}B(\eta))=[\frac{Sc}{4n}-\frac{tSc^{2}}{16n(n+2)^{2}}]h_{t}(j^{-}_{\ast
u}B(\xi),j^{-}_{\ast u}B(\eta))\nonumber
\end{eqnarray}

Similarly, from (\ref{ti}) and  (\ref{kp-})
we get *-Ricci tensors $\rho^{*-}_{1}$ of  $(Z^{-},I_{1}, h_{t})$,
$\rho^{*-}_{2}$ of  $(Z^{-},I_{2}, h_{t})$:

\begin{eqnarray}\label{sro-}
&\rho^{*-}_{1}(j^{-}_{\ast u}A^{*},j^{-}_{\ast u}B^{*})
=[-\frac{t Sc^{2}}{16(n+2)^{2}}+\frac{Sc}{2(n+2)}+\frac{1}{nt}]h_{t}(j^{-}_{\ast u}A^{*},j^{-}_{\ast u}B^{*})\nonumber\\
&\rho^{*-}_{2}(j^{-}_{\ast u}A^{*},j^{-}_{\ast u}B^{*})=
[\frac{t Sc^{2}}{16(n+2)^{2}}-\frac{Sc}{2(n+2)}+\frac{1}{nt}]h_{t}(j^{-}_{\ast u}A^{*},j^{-}_{\ast u}B^{*})\nonumber\\
&\rho^{*-}_{1}(j^{-}_{\ast u}A^{*},j^{-}_{\ast u}B(\xi))=\rho^{*-}_{2}(j^{-}_{\ast u}A^{*},j^{-}_{\ast u}B(\xi))=0\\
&\rho^{*-}_{1}(j^{-}_{\ast u}B(\xi),j^{-}_{\ast u}B(\eta))=
[\frac{Sc(n+1)}{4n(n+2)}]h_{t}(j^{-}_{\ast u}B(\xi),j^{-}_{\ast u}B(\eta))\nonumber\\
&\rho^{*-}_{2}(j^{-}_{\ast u}B(\xi),j^{-}_{\ast
u}B(\eta))=[\frac{tSc^{2}}{16n(n+2)^{2}}+\frac{Sc(n-1)}{4n(n+2)}]h_{t}(j^{-}_{\ast
u}B(\xi),j^{-}_{\ast u}B(\eta))\nonumber
\end{eqnarray}

Now $(iv)$, $(v)$ and $(vi)$ follows from (\ref{ro-}) and (\ref{sro-}).
\end{proof}

\begin{thm}\label{TR}
Let $(M,\sigma,g)$ be a $4n$-dimensional paraquaternionic K\"{a}hler
manifold and $(Z^{+}, P_{i}, h_{t})$, $i = 1, 2$, the reflector spaces
associated. Then:
\\$(i)$ The manifolds $(Z^{+}, P_{i}, h_{t}),\ i = 1,2$,
belong always to  $APH2$ and $APH3$ and are with para-Hermitian
Ricci tensor and with para-Hermitian *-Ricci tensor;
\\$(ii)$ The manifolds $(Z^{+}, P_{1}, h_{t})$ belong to $APH1$ iff $Sc = 0$ or $Sc =-
\frac{4(n + 2)}{t}$;
\\$(iii)$ The manifolds $(Z^{+}, P_{2}, h_{t})$ belong to $APH1$ iff $Sc = 0$;
\\$(iv)$ The manifolds $(Z^{+}, P_{i}, h_{t}), i = 1, 2$  are Einstein iff
\[ Sc = -\frac{ 4(n + 2) }{ t }\ \text{\rm or}\ Sc = -\frac{4(n+2)}{ (n+1)t
};\]
\\$(v)$ The manifolds $(Z^{+}, P_{1}, h_{t})$ are *-Einstein iff
\[ Sc = -\frac{4(n + 2)}{ t}\ \text{\rm or}\ Sc = \frac{4(n+2)}{ nt};\]
\\$(vi)$ The manifolds $(Z^{+}, P_{2}, h_{t})$ are *-Einstein iff
\[ Sc=\frac{2(n+2)}{ (n-1)t }[1-3n -\sqrt{ 9n^{2}-10n+5 }] \]
or
\[ Sc=\frac{2(n+2)}{ (n-1)t }[1-3n +\sqrt{ 9n^{2}-10n+5 }]. \]
\end{thm}

\begin{proof}
$(i)$  This statement follows from (\ref{p1}), (\ref{ror}), (\ref{m1}), (\ref{tp}) and (\ref{kp+}) after some long
but straightforward computations.\\
$(ii)+(iii)$ From (\ref{tp}) and (\ref{kp+}) we obtain

\begin{eqnarray}\label{+aph1}
&K^{+}(j^{+}_{\ast u}A^{*},j^{+}_{\ast u}B^{*},P_i(j^{+}_{\ast u}C^{*}),P_i(j^{+}_{\ast u}D^{*}))
=-\frac{t}{n}(J^{0}_{1}A, B)(J^{0}_{1}C, D)\nonumber\\
&K^{+}(j^{+}_{\ast u}A^{*},j^{+}_{\ast u}B^{*},P_i(j^{+}_{\ast u}C^{*}),P_i(j^{+}_{\ast u}B(\xi)))=0\nonumber\\
&K^{+}(j^{+}_{\ast u}A^{*},j^{+}_{\ast u}B(\xi),P_i(j^{+}_{\ast u}B^{*}),P_i(j^{+}_{\ast u}B(\eta)))=\nonumber\\
&=(\varepsilon_i\frac{t^{2}Sc^{2}}{64n(n+2)^{2}}+\varepsilon_i\frac{t Sc}{8n(n+2)})(A,B)g(X,Y)+\varepsilon_i\frac{t^{2}Sc^{2}}{64n(n+2)^{2}}(J^{0}_{1}A,
B)g(J_{1}X,Y)\\
&K^{+}(j^{+}_{\ast u}B(\xi),j^{+}_{\ast u}B(\eta),P_i(j^{+}_{\ast u}B(\zeta)),P_i(j^{+}_{\ast u}A^{*}))=0\nonumber\\
&K^{+}(j^{+}_{\ast u}B(\xi),j^{+}_{\ast u}B(\eta),P_i(j^{+}_{\ast u}B(\zeta)),P_i(j^{+}_{\ast u}B(\tau)))=-R(X,Y,Z,T)-\nonumber\\
&-\frac{t Sc^{2}}{64n(n+2)^{2}}\{g(Y,J_{2}Z)g(X,J_{2}T)-g(Y,J_{3}Z)g(X,J_{3}T)-\nonumber\\
&-g(X,J_{2}Z)g(Y,J_{2}T)+g(X,J_{3}Z)g(Y,J_{3}T)\}+\nonumber\\
&+(\frac{Sc}{4n(n+2)}+\frac{2t Sc^{2}}{64n(n+2)^{2}})(-g(X,J_{2}Y)g(Z,J_{2}T)+g(X,J_{3}Y)g(Z,J_{3}T))\nonumber
\end{eqnarray}
for $i=1,2, \varepsilon_1=-1,\ \varepsilon_2=+1,\ u \in P(M),\ A,B,C,D\in m_1,\ \xi,\eta,\zeta,\tau\in
\mathbb{R}^{4n}$ and $u(\xi)=X,\ u(\eta)=Y,\ u(\zeta)=Z,\ u(\tau)=T$.

Now the conclusion follows from (\ref{kp+}) and (\ref{+aph1}).
\\
$(iv)+(v)+(vi)$ Using (\ref{kp+}) we obtain the
following formulas for the Ricci tensors  $\rho^{+}$ of  $(Z^{+}, h_{t})$:

\begin{eqnarray}\label{ro+}
&\rho^{+}(j^{+}_{\ast u}A^{*},j^{+}_{\ast u}B^{*})
=[-\frac{tSc^{2}}{16(n+2)^{2}}-\frac{1}{nt}]h_{t}(j^{+}_{\ast u}A^{*},j^{+}_{\ast u}B^{*})\nonumber\\
&\rho^{+}(j^{+}_{\ast u}A^{*},j^{+}_{\ast u}B(\xi))=0\\
&\rho^{+}(j^{+}_{\ast u}B(\xi),j^{+}_{\ast
u}B(\eta))=[\frac{Sc}{4n}+\frac{t Sc^{2}}{16n(n+2)^{2}}]h_{t}(j^{+}_{\ast
u}B(\xi),j^{+}_{\ast u}B(\eta)).\nonumber
\end{eqnarray}

Similarly, from (\ref{tp}) and (\ref{kp+})
we get *-Ricci tensors $\rho^{*+}_{1}$ of
$(Z^{+},P_{1}, h_{t})$ and $\rho^{*+}_{2}$ of  $(Z^{+},P_{2},
h_{t})$:

\begin{eqnarray}\label{sro+}
&\rho^{*+}_{1}(j^{+}_{\ast u}A^{*},j^{+}_{\ast u}B^{*})
=[-\frac{tSc^{2}}{16(n+2)^{2}}-\frac{Sc}{2(n+2)}+\frac{1}{nt}]h_{t}(j^{+}_{\ast u}A^{*},j^{+}_{\ast u}B^{*})\nonumber\\
&\rho^{*+}_{2}(j^{+}_{\ast u}A^{*},j^{+}_{\ast u}B^{*})
=[\frac{t Sc^{2}}{16(n+2)^{2}}+\frac{Sc}{2(n+2)}+\frac{1}{nt}]h_{t}(j^{+}_{\ast u}A^{*},j^{+}_{\ast u}B^{*})\nonumber\\
&\rho^{*+}_{1}(j^{+}_{\ast u}A^{*},j^{+}_{\ast u}B(\xi))=\rho^{*+}_{2}(j^{+}_{\ast u}A^{*},j^{+}_{\ast u}B(\xi))=0\\
&\rho^{*+}_{1}(j^{+}_{\ast u}B(\xi),j^{+}_{\ast u}B(\eta))
=[-\frac{Sc(n+1)}{4n(n+2)}]h_{t}(j^{+}_{\ast u}B(\xi),j^{+}_{\ast u}B(\eta))\nonumber\\
&\rho^{*+}_{2}(j^{+}_{\ast u}B(\xi),j^{+}_{\ast
u}B(\eta))=[\frac{t Sc^{2}}{16n(n+2)^{2}}-\frac{Sc(n-1)}{4n(n+2)}]h_{t}(j^{+}_{\ast
u}B(\xi),j^{+}_{\ast u}B(\eta)).\nonumber
\end{eqnarray}

Now $(iv)$, $(v)$ and $(vi)$ follows from (\ref{ro+}) and (\ref{sro+}).
\end{proof}

\begin{rem}
We note that the corresponding result of the Theorem \ref{TT} in
quaternionic setting was obtained in \cite{AGI} and the assertions
stated at items $(iv)$ in Theorems \ref{TT} and \ref{TR} were also
proved, but using a different method, by Alekseevsky and Cort\'{e}s
(compare with \cite[Corollary 6, page 126]{AC}).

\end{rem}

\section{Mixed 3-Sasakian structures in a $\text{SO}(2,1)$-principal bundle over a paraquaternionic K\"{a}hler manifold}

\begin{defn}
        Let $M$ be a smooth manifold equipped with a triple
        $(\varphi,\xi,\eta)$, where $\varphi$ is a field  of endomorphisms
        of the tangent spaces, $\xi$ is a vector field and $\eta$ is a
        1-form on $M$. If we have:
\begin{equation}\label{V1}
        \varphi^2=\tau(-I+\eta\otimes\xi),\ \ \  \eta(\xi)=1
\end{equation}
        then we say that:

        $(i)$ $(\varphi,\xi,\eta)$ is an \emph{almost contact
        structure} on $M$, if $\tau=1$ (\cite{SAS}).

        $(ii)$ $(\varphi,\xi,\eta)$ is an \emph{almost paracontact
        structure} on $M$, if $\tau=-1$ (\cite{SAT}).
\end{defn}

We remark that many authors also include in the above definition the
conditions that
\[
\varphi\xi=0,\ \eta\circ\varphi=0,
\]
although these are deducible from (\ref{V1}) (see \cite{BLR}).

\begin{defn}\cite{CP}
A \emph{mixed 3-structure} on a smooth manifold $M$ is a triple of
structures $(\varphi_\alpha,\xi_\alpha,\eta_\alpha)$,
$\alpha\in\{1,2,3\}$, which are almost paracontact structures for
$\alpha=1,2$ and almost contact structure for $\alpha=3$, satisfying
the following conditions:
\begin{equation}\label{V3}
\eta_\alpha(\xi_\beta)=0,
\end{equation}
\begin{equation}\label{V4}
\varphi_\alpha(\xi_\beta)=\tau_\beta\xi_\gamma,\ \ \varphi_\beta(\xi_\alpha)=-\tau_\alpha\xi_\gamma,\\
\end{equation}
\begin{equation}\label{V5}
\eta_\alpha\circ\varphi_\beta=-\eta_\beta\circ\varphi_\alpha=
\tau_\gamma\eta_\gamma\,,\\
\end{equation}
\begin{equation}\label{V6}
\varphi_\alpha\varphi_\beta-\tau_\alpha\eta_\beta\otimes\xi_\alpha=
-\varphi_\beta\varphi_\alpha+\tau_\beta\eta_\alpha\otimes\xi_\beta=
\tau_\gamma\varphi_\gamma\,,
\end{equation}
where $(\alpha,\beta,\gamma)$ is an even permutation of $(1,2,3)$
and $\tau_1=\tau_2=-\tau_3=-1$.

Moreover, if a manifold $M$ with a mixed 3-structure
$(\varphi_\alpha,\xi_\alpha,\eta_\alpha)_{\alpha=\overline{1,3}}$
 admits a semi-Riemannian metric $g$ such that:
\begin{equation}\label{V7}
g(\varphi_\alpha X, \varphi_\alpha Y)=\tau_\alpha
[g(X,Y)-\varepsilon_\alpha\eta_\alpha(X)\eta_\alpha(Y)],
\end{equation}
for all $X,Y\in\Gamma(TM)$ and $\alpha=1,2,3$, where
$\varepsilon_\alpha=g(\xi_\alpha,\xi_\alpha)=\pm1$, then we say that
$M$ has a \emph{metric mixed 3-structure} and $g$ is called a
\emph{compatible metric}.
\end{defn}

\begin{rem}
From (\ref{V7}) we obtain
\begin{equation}\label{V8}
\eta_\alpha(X)=\varepsilon_\alpha g(X,\xi_\alpha),\ g(\varphi_\alpha
X, Y)=- g(X,\varphi_\alpha Y)
\end{equation}
for all $X,Y\in\Gamma(TM)$ and $\alpha=1,2,3$.

Note that if
$(M,(\varphi_\alpha,\xi_\alpha,\eta_\alpha)_{\alpha=\overline{1,3}},g)$
is a manifold with a metric mixed 3-structure then from (\ref{V8})
it follows
\[
g(\xi_1,\xi_1)=g(\xi_2,\xi_2)=-g(\xi_3,\xi_3).
\]

Hence the vector fields $\xi_1$ and $\xi_2$ are both either
space-like or time-like and these force the causal character of the
third vector field $\xi_3$. We may therefore distinguish between
\emph{positive} and \emph{negative metric mixed 3-structures},
according as $\xi_1$ and $\xi_2$ are both space-like, or both
time-like vector fields. Because one can check that, at each point
of $M$,
 there always exists a pseudo-orthonormal frame field given by
$\{(E_i,\varphi_1 E_i, \varphi_2 E_i, \varphi_3
E_i)_{i=\overline{1,n}}\,, \xi_1, \xi_2, \xi_3\}$ we conclude that
the dimension of the manifold is $4n+3$ and the signature of $g$ is
$(2n+1,2n+2)$, where we put first the minus signs, if the metric
mixed 3-structure is positive (\emph{i.e.}
$\varepsilon_1=\varepsilon_2=-\varepsilon_3=1$), or the signature of
$g$ is $(2n+2,2n+1)$, if the metric mixed 3-structure is negative
(\emph{i.e.} $\varepsilon_1=\varepsilon_2=-\varepsilon_3=-1$).
\end{rem}

\begin{defn}\cite{CP}
Let $M$ be a manifold endowed with a (positive/negative) metric
mixed 3-–structure
$((\varphi_\alpha,\xi_\alpha,\eta_\alpha)_{\alpha=\overline{1,3}},g)$.
This structure is said to be:

$(i)$ a (positive/negative) \emph{metric mixed 3--contact structure}
if $d\eta_\alpha=\Phi_\alpha$, for each $\alpha\in\{1,2,3\}$, where
$\Phi_\alpha$ is the fundamental 2--form defined by
$\Phi_\alpha(X,Y):=g(X,\varphi_\alpha Y)$.

$(ii)$ a (positive/negative) \emph{mixed 3–-K-–contact structure} if
$((\varphi_3,\xi_3,\eta_3),g)$ is a K–-contact structure and
$((\varphi_1,\xi_1,\eta_1),g)$, $((\varphi_2,\xi_2,\eta_2),g)$ are
para–-K–-contact structures, i.e.
$\nabla\xi_\alpha=-\varepsilon_\alpha\varphi_\alpha$, for each
$\alpha\in\{1,2,3\}$, where $\nabla$ is Levi-Civita connection of
$g$.

$(iii)$ a (positive/negative) \emph{mixed 3-Sasakian structure} if
$(\varphi_3,\xi_3,\eta_3,g)$ is a Sasakian structure and
$(\varphi_1,\xi_1,\eta_1,g)$, $(\varphi_2,\xi_2,\eta_2,g)$ are
para-Sasakian structures, \emph{i.e.}

\begin{equation}\label{m3s}
(\nabla_X\varphi_\alpha)
Y=\tau_\alpha[g(X,Y)\xi_\alpha-\varepsilon_\alpha\eta_\alpha(Y)X]
\end{equation}

for all $X,Y\in\Gamma(TM)$ and $\alpha\in\{1,2,3\}$.
\end{defn}

\begin{rem}\label{r1}
Note that in fact mixed metric 3-–contact, mixed 3–-K–-contact and
mixed 3–-Sasakian structures define the same class of manifolds, as
proved in \cite{CP}. Moreover, like their Riemannian counterparts,
mixed 3-Sasakian structures are Einstein, but now the scalar
curvature can be either positive or negative:
\end{rem}

\begin{thm}\label{cp}\cite{CP,IV} Any $(4n+3)-$dimensional manifold endowed with a
mixed $3$-Sasakian structure is an Einstein space with Einstein
constant $\lambda=(4n+2)\varepsilon$, with $\varepsilon=\mp1$,
according as the metric mixed 3-structure is positive or negative,
respectively.
\end{thm}

\begin{rem} Several examples of manifolds endowed with mixed 3-Sasakian structures
can be found in \cite{IVV}. We note that the unit pseudo-sphere
$S^{4n+3}_{2n+2}\subset\mathbb{R}^{4n+4}_{2n+2}$ is the canonical
example of manifold with negative mixed 3-Sasakian structure, while
the pseudo-hyperbolic space
$H^{4n+3}_{2n+1}\subset\mathbb{R}^{4n+4}_{2n+2}$ can be endowed with
a canonical positive mixed 3-Sasakian structure.
\end{rem}

Let $(M,\sigma,g)$ be a paraquaternionic K\"{a}hler manifold and let
$P$ be the bundle associated with $(M,\sigma,g)$. That is, $P$ is
the bundle whose transition functions and structure group are the
same as $\sigma$, but whose fibre is $\text{SO}(2,1)$. The Lie
algebra of $\text{SO}(2,1)$ is $\text{so}(2,1)$ and we consider two
of its bases: $\mathcal{B}^+=\{e_1,e_2,e_3\}$ and
$\mathcal{B}^-=\{-e_1,-e_2,-e_3\}$, where
\[
e_1=\left(%
\begin{array}{ccc}
  0 & 0 & 0 \\
  0 & 0 & 2 \\
  0 & 2 & 0 \\
\end{array}%
\right),\ e_2=\left(%
\begin{array}{ccc}
  0 & 0 & 2 \\
  0 & 0 & 0 \\
  2 & 0 & 0 \\
\end{array}%
\right),\ e_3=\left(%
\begin{array}{ccc}
  0 & -2 & 0 \\
  2 & 0 & 0 \\
  0 & 0 & 0 \\
\end{array}%
\right).
\]

It can be easily verified that
\begin{equation}\label{V15}
[e_1,e_2]=2e_3,\ [e_2,e_3]=-2e_1,\ [e_3,e_1]=-2e_2,
\end{equation}
and
\begin{equation}\label{V16}
[-e_1,-e_2]=-2(-e_3),\ [-e_2,-e_3]=2(-e_1),\ [-e_3,-e_1]=2(-e_2).
\end{equation}

On the other hand, their exp are
\[
{\rm exp}(te_1)=\left(%
\begin{array}{ccc}
  1 & 0 & 0 \\
  0 & {\rm cosh}2t & {\rm sinh}2t \\
  0 & {\rm sinh}2t & {\rm cosh}2t \\
\end{array}%
\right),
\]
\[
{\rm exp}(te_2)=\left(%
\begin{array}{ccc}
  {\rm cosh}2t & 0 & {\rm sinh}2t \\
  0 & 1 & 0 \\
  {\rm sinh}2t & 0 & {\rm cosh}2t \\
\end{array}%
\right),
\]
\[
{\rm exp}(te_2)=\left(%
\begin{array}{ccc}
  {\rm cos}2t & -{\rm sin}2t & 0 \\
   {\rm sin}2t & {\rm cos}2t & 0 \\
  0 & 0 & 1 \\
\end{array}%
\right).
\]

Then, with respect to the basis $\mathcal{B}^+$, we have
\[
{\rm ad}(e_1)=\left(%
\begin{array}{ccc}
  0 & 0 & 0 \\
  0 & 0 & 2 \\
  0 & 2 & 0 \\
\end{array}%
\right),\ {\rm ad}(e_2)=\left(%
\begin{array}{ccc}
  0 & 0 & -2 \\
  0 & 0 & 0 \\
  -2 & 0 & 0 \\
\end{array}%
\right),\ {\rm ad}(e_3)=\left(%
\begin{array}{ccc}
  0 & 2 & 0 \\
  -2 & 0 & 0 \\
  0 & 0 & 0 \\
\end{array}%
\right)
\]
and with respect to the basis $\mathcal{B}^-$, we have
\[
{\rm ad}(-e_1)=\left(%
\begin{array}{ccc}
  0 & 0 & 0 \\
  0 & 0 & -2 \\
  0 & -2 & 0 \\
\end{array}%
\right), {\rm ad}(-e_2)=\left(%
\begin{array}{ccc}
  0 & 0 & 2 \\
  0 & 0 & 0 \\
  2 & 0 & 0 \\
\end{array}%
\right), {\rm ad}(-e_3)=\left(%
\begin{array}{ccc}
  0 & -2 & 0 \\
  2 & 0 & 0 \\
  0 & 0 & 0 \\
\end{array}%
\right).
\]

For each neighborhood $U$ of $(M,\sigma,g)$, we define a
$\text{so}(2,1)$-valued 1-form on $U$ by
\begin{equation}\label{V17}
\omega_U=\sum_{i=1}^{3}\omega_ie_i.
\end{equation}

From (\ref{V14}) and (\ref{V17}) we obtain that on $U\cap U'$ we
have
\[
\omega_{U'}(X)={\rm
ad}(S^{-1}_{UU'})\cdot\omega_U(X)+(S_{UU'})_*(X)\cdot S^{-1}_{UU'}
\]
for all vector fields $X$ on $P$, where $(S_{UU'})_*$ denotes the
differential of the mapping $S_{UU'}:U\cap U'\rightarrow
\text{SO}(2,1)$. Therefore there exists a connection form $\omega$
on $P$ such that
\[
\psi^*\omega=\omega_U,
\]
where $\psi$ is a certain local cross-section of $P$ over $U$.

The curvature form $\Omega$ defined by the connection $\omega$ is a
$\text{so}(2,1)$-valued 2-form given by
\[
\Omega(X,Y)=d\omega(X,Y)+\frac{1}{2}[\omega(X),\omega(Y)]
\]
for all vector fields $X,Y$ on $P$. Hence, using (\ref{V15}) we
derive:
\[
\psi^*\Omega=(d\omega_1-\omega_2\wedge
\omega_3)e_1+(d\omega_2-\omega_3\wedge \omega_1)e_2
+(d\omega_3+\omega_1\wedge \omega_2)e_3.
\]

\begin{thm}
The principal bundle $P$ associated with a paraquaternionic
K\"{a}hler manifold $(M,\sigma,g)$ of non-zero scalar curvature can
be endowed with both a positive and a negative mixed 3-Sasakian
structure.
\end{thm}
\begin{proof}
If $\omega_U=\displaystyle\sum_{i=1}^{3}\omega_ie_i$ is the
infinitesimal connection in $P$ considered above, then we can define
a semi-Riemannian metric $\widetilde{g}$ on $P$ by
\begin{equation}\label{V21}
\widetilde{g}=\nu\pi_*g+\sum_{i=1}^{3}\varepsilon_i\omega_i\otimes\omega_i,
\end{equation}
where $\nu$ is the reduced scalar curvature of $M$ and
$(\varepsilon_1,\varepsilon_2,\varepsilon_3)=(-1,-1,1)$ (Case I) or
$(\varepsilon_1,\varepsilon_2,\varepsilon_3)=(1,1,-1)$ (Case II).

We remark that unlike quaternionic setting (see \cite{KO}), the
signature of $\widetilde{g}$ does not depend on the sign of the
scalar curvature. In fact, the signature is $(2n+2,2n+1)$ in the
Case I and $(2n+1,2n+2)$ in the Case II.

Firstly, we consider the Case I. We define the 1-forms
$\eta_\alpha=\omega_\alpha$ for $\alpha=1,2,3$, and let $\xi_1,\
\xi_2,\ \xi_3$ be the fundamental vector fields corresponding to
$e_1,\ e_2,\ e_3$, respectively. Then from (\ref{V15}) we deduce
\[
[\xi_1,\xi_2]=2\xi_3,\ [\xi_2,\xi_3]=-2\xi_1,\ [\xi_3,\xi_1]=-2\xi_2
\]
and
\begin{equation}\label{V23}
\eta_\alpha(\xi_\beta)=\delta_{\alpha\beta},\ \alpha,\beta=1,2,3.
\end{equation}

If $\widetilde{\nabla}$ is the Levi-Civita connection of
$\widetilde{g}$, then we define
\begin{equation}\label{V24}
\varphi_\alpha=-\varepsilon_\alpha\widetilde{\nabla}\xi_\alpha,\
\alpha=1,2,3.
\end{equation}

Next we prove that
$((\varphi_\alpha,\xi_\alpha,\eta_\alpha)_{\alpha=\overline{1,3}},\widetilde{g})$
is a negative mixed 3-Sasakian structure on $P$.

If $p\in P$, then we denote by $T_p^V(P)$ the tangent space of a
fibre at $p$ and by $T_p^H(P)$ its orthogonal complemented space in
$T_pP$. Using now (\ref{V24}) and Koszul formula, we derive
\begin{equation}\label{V27}
\varphi_\alpha(\xi_\alpha)=0,\
\varphi_\alpha(\xi_\beta)=\tau_\beta\xi_\gamma,\ \
\varphi_\beta(\xi_\alpha)=-\tau_\alpha\xi_\gamma.
\end{equation}

From (\ref{V27}) it follows that $T_p^V(P)$ and $T_p^H(P)$ are
invariant under the action of $\varphi_\alpha$, $\alpha=1,2,3$.
Therefore we have:
\begin{equation}\label{V28}
\varphi_\alpha=\varphi^H_\alpha+\tau_\beta\eta_\beta\otimes\xi_\gamma-\tau_\gamma\eta_\gamma\otimes\xi_\beta
\end{equation}
for any even permutation $(\alpha,\beta,\gamma)$ of $(1,2,3)$, where
$\varphi^H_\alpha$ denotes the restricted actions of
$\varphi_\alpha$ on $T_p^H(P)$, for $\alpha=1,2,3$.

Now, from the structure equations (\ref{V13}), we obtain for each
neighborhood $U$ in $M$ and a local cross section $\psi:U\rightarrow
P$ the following relation
\[
(d\omega_\alpha+\tau_\alpha\omega_\beta\wedge
\omega_\gamma)(\psi_*X,\psi_*Y)=\tau_\alpha\nu g(J_\alpha X,Y),
\]
for any even permutation $(\alpha,\beta,\gamma)$ of $(1,2,3)$, where
$\psi_*$ denotes the differential of $\psi$. Since the curvature
form is horizontal, we have
\[(\varphi_\alpha-\tau_\beta\eta_\beta\otimes\xi_\gamma+\tau_\gamma\eta_\gamma\otimes\xi_\beta)(\psi_*X)=
(\psi_*J_\alpha X)^H,\] so we derive
\begin{equation}\label{V30}
\varphi_\alpha(\psi_*X)=(\psi_*J_\alpha
X)^H+\tau_\beta\eta_\beta(\psi_*X)\xi_\gamma-\tau_\gamma\eta_\gamma(\psi_*X)\xi_\beta.
\end{equation}

Using now (\ref{V25}), (\ref{V23}) and (\ref{V30}) we obtain
\begin{eqnarray}
\varphi_\alpha^2(\psi_*X)&=&\varphi_\alpha[(\psi_*J_\alpha
X)^H+\tau_\beta\eta_\beta(\psi_*X)\xi_\gamma-\tau_\gamma\eta_\gamma(\psi_*X)\xi_\beta]\nonumber\\
&=&(\psi_*J^2_\alpha
X)^H-\tau_\beta\tau_\gamma\eta_\gamma(\psi_*X)\xi_\gamma-\tau_\gamma\tau_\beta\eta_\beta(\psi_*X)\xi_\beta\nonumber\\
&=&-\tau_\alpha(\psi_*X)^H-\tau_\beta\tau_\gamma[\eta_\gamma(\psi_*X)\xi_\gamma+\eta_\beta(\psi_*X)\xi_\beta]\nonumber
\end{eqnarray}
and since $\tau_\alpha\tau_\beta\tau_\gamma=1$, we deduce
\begin{eqnarray}\label{V31}
\varphi_\alpha^2(\psi_*X)&=&-\tau_\alpha[(\psi_*X)^H+\eta_\gamma(\psi_*X)\xi_\gamma+\eta_\beta(\psi_*X)\xi_\beta]\nonumber\\
&=&-\tau_\alpha[\psi_*X-\eta_\alpha(\psi_*X)\xi_\alpha].
\end{eqnarray}

From (\ref{V23}) and (\ref{V31}) we conclude that
$(\varphi_1,\xi_1,\eta_1)$ and $(\varphi_2,\xi_2,\eta_2)$ are almost
paracontact structures on $P$, while $(\varphi_3,\xi_3,\eta_3)$ is
an almost contact structure on $P$. Now we are able to prove the
compatibility conditions between these structures.

Using again (\ref{V25}), (\ref{V23}) and (\ref{V30}) we obtain
\begin{eqnarray}
\varphi_\alpha\varphi_\beta(\psi_*X)&=&\varphi_\alpha[(\psi_*J_\beta
X)^H+\tau_\gamma\eta_\gamma(\psi_*X)\xi_\alpha-\tau_\alpha\eta_\alpha(\psi_*X)\xi_\gamma]\nonumber\\
&=&(\psi_*J_\alpha J_\beta X)^H+\tau_\alpha
\tau_\gamma\eta_\alpha(\psi_*X)\xi_\beta\nonumber\\
&=&\tau_\gamma[(\psi_*J_\gamma
X)^H+\tau_\alpha\eta_\alpha(\psi_*X)\xi_\beta]\nonumber\\
&=&\tau_\gamma[\varphi_\gamma(\psi_*X)+\tau_\beta\eta_\beta(\psi_*X)\xi_\alpha]\nonumber
\end{eqnarray}
and taking into account that $\tau_\alpha\tau_\beta\tau_\gamma=1$,
we deduce
\[
\varphi_\alpha\varphi_\beta(\psi_*X)=\tau_\gamma\varphi_\gamma(\psi_*X)+\tau_\alpha(\psi_*X)\xi_\alpha.
\]
Similarly we find
\[
\varphi_\beta\varphi_\alpha(\psi_*X)=-\tau_\gamma\varphi_\gamma(\psi_*X)+\tau_\beta(\psi_*X)\xi_\beta.
\]
and so the compatibility condition (\ref{V6}) is verified.
Analogously we find that the compatibility condition (\ref{V5}) is
checked. Moreover, from (\ref{V23}) and (\ref{V27}) we have that
(\ref{V3}) and (\ref{V4}) are also true.

On the other hand, using (\ref{V21}) and (\ref{V30}) we derive
\begin{eqnarray}
\widetilde{g}(\varphi_\alpha\psi_*X,\varphi_\alpha\psi_*Y)&=&\nu
g(J_\alpha X,J_\alpha
Y)+\varepsilon_\beta\eta_\gamma(\psi_*X)\eta_\gamma(\psi_*Y)+
\varepsilon_\gamma\eta_\beta(\psi_*X)\eta_\beta(\psi_*Y)\nonumber\\
&=&\tau_\alpha[\nu
g(X,Y)+\varepsilon_\beta\eta_\beta(\psi_*X)\eta_\beta(\psi_*Y)+\varepsilon_\gamma\eta_\gamma(\psi_*X)\eta_\gamma(\psi_*Y)]\nonumber\\
&=&\tau_\alpha[\widetilde{g}(\psi_*X,\psi_*Y)-\varepsilon_\alpha\eta_\alpha(\psi_*X)\eta_\alpha(\psi_*Y)].\nonumber
\end{eqnarray}
Hence the metric $\widetilde{g}$ is compatible with
$(\varphi_\alpha,\xi_\alpha,\eta_\alpha)$, for $\alpha={1,2,3}$.

Therefore, taking account of (\ref{V24}), we deduce that
$((\varphi_\alpha,\xi_\alpha,\eta_\alpha)_{\alpha=\overline{1,3}},\widetilde{g})$
is a negative mixed 3–-K-–contact structure in $P$, and so, via
Remark \ref{r1}, is a negative mixed 3–-Sasakian structure in $P$.

In the Case II, we consider $\xi'_1,\ \xi'_2,\ \xi'_3$ to be the
fundamental vector fields corresponding to $-e_1,\ -e_2,\ -e_3$,
respectively, and we define \[\eta'_\alpha=-\omega_\alpha,\
\varphi'_\alpha=-\varepsilon_\alpha\widetilde{\nabla}\xi'_\alpha,\]
for $\alpha=1,2,3$. Then, it can be proved similarly as in Case I
that
$((\varphi'_\alpha,\xi'_\alpha,\eta'_\alpha)_{\alpha=\overline{1,3}},\widetilde{g})$
is a positive mixed 3–-Sasakian structure in $P$.
\end{proof}

\begin{rem}
We note that in \cite{CP2}, the authors investigated the possible
projectability onto paraquaternionic structures of metric mixed
3–-contact structures defined on the total space of a $G$–-principal
bundle. In fact they proved that if $P(M,G,\pi)$ is a $G$–-principal
bundle, where the total space $P$ is endowed with a metric mixed
3–-contact structure
$((\varphi_\alpha,\xi_\alpha,\eta_\alpha)_{\alpha=\overline{1,3}},g)$,
$G$ is a Lie group acting on the right on $P$ by isometries having
Lie algebra isomorphic to $\text{so}(2,1)$ and such that the vector
fields $\xi_1,\xi_2,\xi_3$ are fundamental vector fields, then the
metric mixed 3–-contact structure projects via $\pi$ onto a
paraquaternionic K\"{a}hler structure $(\sigma,h)$ on the base space
$M$ and $\pi:(P,g) \rightarrow (M,h)$ is a pseudo–-Riemannian
submersion. Moreover, if the dimension of $P$ is $(4n+3)$, then one
has
\begin{equation}\label{ric}
\rho^M=\varepsilon4(n+2)h
\end{equation}
were $\rho^M$ denotes the Ricci curvature tensor of $M$ and
$\varepsilon=\mp1$, according as the mixed 3–-structure is positive
or negative.
\end{rem}
Let now $X,Y$ and $U,V$ be basic and vertical vector fields
throughout. The metric $g$ splits up as
\[ g(E,F)=g(X+U,Y+V)=g(X,Y)+g(U,V),\]
where $E=X+U, F=Y+V$ are vector fields on $P$. Then the canonical
variation of the metric $g$ for $t>0$, is a pseudo-Riemannian metric
$g_t$ on $P$ defined by (see \cite{BFLM})
\begin{equation}\label{cvd}
g_t (X+U,Y+V)=g(X,Y)+tg(U,V).
\end{equation}

Similarly as in the Riemannian case \cite{De}, it can be proved that
the curvature tensor $R_t$ of the canonical variation $g_t$
satisfies
\begin{eqnarray}\label{rt}
R_t(E,F,G,H)&=tR(E,F,G,H)+(1-t)R^N(X_*,Y_*,Z_*,Z'_*)+\\
&+t(1-t)(g(S_EH,S_FG)-g(S_EG,S_FH))\nonumber
\end{eqnarray}
where $E,F,G,H$ are vector fields on $P$, $X_*,Y_*,Z_*,Z'_*$ are the
corresponding vector fields on $M$ of the basic vector fields
$X,Y,Z,Z'$ and $S_EF=T_UV+A_XV+A_YU$, $E=X+U$, $F=Y+V$, $G=Z+W$,
$H=Z'+W'$, $T$ and $A$ being the O'Neill tensors of the submersion
(see \cite{FIP}). In this setting, we can state now the following

\begin{thm}
If  $\pi:(P,(\varphi_\alpha,\xi_\alpha,\eta_\alpha)_{\alpha=\overline{1,3}},g)
\rightarrow (M,\sigma,h)$ is a pseudo–-Riemannian submersion as in
the previous remark, then there exist a unique positive $t\neq1$
such that $g_t$ is also Einstein.
\end{thm}

\begin{proof}
We denote by $R$ the curvature tensor of manifold $P$ and by $\rho$,
$\rho^M$  the Ricci curvature tensor of manifold $P$, respectively
$M$. Using (\ref{V8}) and (\ref{m3s}) we obtain that
\begin{equation}\label{rm3s}
R(E,F)\xi_\alpha=\tau_\alpha(\eta_\alpha(F)E-\eta_\alpha(E)F),\
\alpha=1,2,3,
\end{equation}
where $E,F$ are vector fields on $P$.

Now using (\ref{cvd}), (\ref{rt}), (\ref{rm3s}) and the definitions
of mixed $3$-Sasakian and paraquaternionic K\"{a}hler structures we
obtain that the components of the Ricci tensor $\rho_t$ of the
canonical variation $g_t$ are given by
\begin{eqnarray}
\rho_t(U,V)&=&t \rho(U,V)+\{ -4n\varepsilon t^2+(4n \varepsilon+2\varepsilon)t-2\varepsilon\}g(U,V),\nonumber\\
\rho_t(U,X)&=&t \rho(U,X),\nonumber\\
\rho_t(X,Y)&=&t \rho(X,Y)+(1-t)\rho^M (X_*,Y_*),\nonumber
\end{eqnarray}
where $X,Y$ and $U,V$ are basic and vertical vector fields on $P$,
$X_*,Y_*$ are the vector fields on $M$ corresponding to the basic
vector fields $X,Y$,  $\varepsilon=\mp1$ according as the metric
mixed 3-structure is positive or negative, respectively.

From Theorem \ref{2.1}, Theorem \ref{cp} and (\ref{ric}) we derive
\begin{eqnarray}\label{rict}
\rho_t(U,V)&=&\{-4n\varepsilon t^2+(8n \varepsilon+4\varepsilon)t-2\varepsilon\}g(U,V),\nonumber\\
\rho_t(U,X)&=&0,\\
\rho_t(X,Y)&=&\{-6\varepsilon t +
4n\varepsilon+8\varepsilon\}g(X,Y),\nonumber
\end{eqnarray}

Now, from (\ref{rict}) we deduce easily that $g_t$ is Einstein if
and only if $t=1$ or $t=\frac{2n+5}{2n}$ and the conclusion follows.

\end{proof}

Gabriel Eduard V\^{I}LCU \\
         Petroleum-Gas University of Ploie\c sti,\\
         Department of Mathematics and Computer Science,\\
         Bulevardul Bucure\c sti, Nr. 39, Ploie\c sti 100680, Romania\\
         E-mail: gvilcu@upg-ploiesti.ro\\
         and\\
         University of Bucharest,
         Faculty of Mathematics and Computer Science,\\
         Research Center in Geometry, Topology and Algebra\\
         Str. Academiei, Nr. 14, Sector 1, Bucharest 70109, Romania\\
         E-mail: gvilcu@gta.math.unibuc.ro\\

Rodica Cristina VOICU \\
         University of Bucharest,
         Faculty of Mathematics and Computer Science\\
         Research Center in Geometry, Topology and Algebra\\
         Str. Academiei, Nr. 14, Sector 1, Bucharest 72200, Romania\\
         E-mail: rcvoicu@gmail.com
\end{document}